\documentclass[12pt]{amsart}
\usepackage{times}
\usepackage[T1]{fontenc}
\usepackage[utf8]{inputenc}
\usepackage{lmodern}
\usepackage{color}
\usepackage[english]{babel}
\usepackage{geometry}
\usepackage{relsize}
\usepackage{booktabs}
\usepackage{mathtools}
\usepackage{amsthm}
\usepackage{amsmath,amssymb}
\usepackage{emptypage}
\usepackage{xfrac}    
\usepackage{faktor}
\usepackage{graphicx}
\usepackage{listings}
\usepackage{lipsum}
\usepackage{setspace}
\usepackage{lipsum}
\usepackage{newlfont}
\usepackage[mathscr]{eucal}	
\usepackage{physics}
\usepackage{tikz}
\usepackage{url}
\usepackage{hyperref}
\usetikzlibrary{positioning}
\usepackage{xcolor}
\usepackage[new]{old-arrows}
\usepackage{multirow}

\theoremstyle{plain}
\newtheorem{theorem}{Theorem}[section]

\newtheorem{proposition}[theorem]{Proposition}

\theoremstyle{definition}
\newtheorem{definition}[theorem]{Definition}

\newtheorem{example}[theorem]{Example}

\newtheorem{remark}{Remark}[section]
\usepackage{setspace}
\usepackage[signatures,write]{frontespizio}

\usepackage{fancyhdr}

\newcommand{\bigslant}[2]{{\raisebox{.2em}{$#1$}\left/\raisebox{-.2em}{$#2$}\right.}}
\newcommand{\CP}{\mathbb{C}\mathbb{P}} 
\newcommand{\torus}{\mathbb{T}} 
\newcommand{\Sph}{\mathbb{S}} 
\newcommand{\Z}{\mathbb{Z}} 
\newcommand{\R}{\mathbb{R}} 
\newcommand{\C}{\mathbb{C}} 

\newcommand{\quat}{\mathbb{H}}
\newcommand{\g}{\mathfrak{g}}

\AtBeginDocument
{
\setlength\abovedisplayskip{.6\baselineskip}
\setlength\abovedisplayshortskip{.6\baselineskip}
\setlength\belowdisplayskip{.5\baselineskip}
\setlength\belowdisplayshortskip{.5\baselineskip}
}
\begin{document}

\title[A mapping tori construction ]{A mapping tori construction of strong HKT and generalized hyperk\"ahler manifolds} 

\author{Beatrice Brienza}
\address[Beatrice Brienza]{Dipartimento di Matematica ``G. Peano'', Universit\`{a} degli studi di Torino \\
Via Carlo Alberto 10\\
10123 Torino, Italy}
\email{beatrice.brienza@unito.it}

\author{Anna Fino}
\address[Anna Fino]{Dipartimento di Matematica ``G. Peano'', Universit\`{a} degli studi di Torino \\
Via Carlo Alberto 10\\
10123 Torino, Italy\\
\& Department of Mathematics and Statistics, Florida International University\\
Miami, FL 33199, United States}
\email{annamaria.fino@unito.it, afino@fiu.edu}

\author{Gueo Grantcharov}
\address[Gueo Grantcharov]{Department of Mathematics and Statistics \\
Florida International University\\
Miami, FL 33199, United States}
\email{grantchg@fiu.edu}

\keywords{Mapping tori, HKT structure, generalized hyperk\"ahler structure}

\subjclass[2024]{53C26}

\maketitle

\begin{abstract}  In the present paper we provide a construction via  mapping tori  of  (non Bismut flat) strong HKT and generalized hyperk\"ahler structures on compact manifolds. The skew-symmetric torsion is parallel, but the manifolds are  not  a product of a  hyperk\"ahler manifold and a compact Lie group.

\end{abstract}

\vskip.6cm

\centerline {\it  Dedicated to  Paul Gauduchon  for his 80th birthday}

\vskip.6cm

\section{Introduction}

A HKT structure on a manifold is a quadruple of three anti-commuting complex structures and a compatible metric, with all structures parallel with respect to a (necessarily unique)  connection with skew-symmetric torsion. When the torsion 3-form is closed the structure is called strong HKT. The HKT and strong HKT structures appeared first in String theory as the structures on the target space of a (0,4)-supersymmetric sigma model with Wess-Zumino term (see \cite{HP, GPS}). They also appear in superconformal quantum mechanics (\cite{MiS}) and the geometry of black hole moduli spaces \cite{GuP}.

 In mathematics, the HKT structures were considered first in \cite{GP} and are regarded as  the closest quaternionic analogs of K\"ahler structures. They admit a local potentials \cite{BS}. When the canonical bundle of one of the complex structures is trivial, HKT structures lead to a form of Hodge Theory \cite{Ve} and a characterization in real 8-dimensional case \cite{GLV} similar to the characterization of compact complex surfaces admitting  a K\"ahler metric. They provide a natural venue for a quaternionic Monge-Ampere equation \cite{AV,Alekser, BGV}. The torsion of the HKT connection coincides of the torsion of a connection considered by Bismut  \cite{Bi}  and Strominger  \cite{Strominger} is often called Bismut or Bismut-Strominger connection. In particular the HKT structures on $4n$-manifolds  are characterized as connection  with skew-symmetric torsion and local holonomy in $Sp(n)$.
 
 The strong HKT structures also appear in physics as part of the structure on the target space of (4,4) SUSY sigma model with B-field (or Wess-Zumino term). Such symmetry leads to a pair of strong HKT structures with an opposite torsion, which in mathematics are known as generalized hyperk\"ahler structures  \cite{Bredthauer}. In a recent preprint \cite{W} the strong HKT and generalized hyperkahler structures are studied in relation to AdS/CFT duality.
 
 The strong HKT structures are also pluriclosed and Bismut Hermitian-Einstein (or BHE), and an open question is whether there exist such a structure on a compact manifold which is not a local product of  a K\"ahler space and a Bismut flat space (\cite{GJS, BFG}).  Note that a  remarkable result of Gauduchon and Ivanov   \cite{GI}  states that  in complex dimension 2  the only BHE manifolds are Bismut flat, and in fact quotients of
the standard Bismut flat Hopf surface and recently a classification of compact   BHE manifolds in complex dimension 3 is given in   \cite{ALS}.  The  similar question for strong HKT structures  is considered in \cite{W}.

In this note we provide,  using the previous results in \cite{BFG},   a construction of strong HKT and generalized K\"ahler structures on compact  manifolds which are not a  product  of hyperk\"ahler manifold and a compact Lie group. According to our knowledge these are the first  examples of non Bismut flat  compact HKT  and generalized K\"ahler  manifolds, but the Bismut torsion is still parallel, so they do not provide an answer to the questions above.

The paper is organized as follows.  In Section 2 we briefly review some relevant facts about the  HKT manifolds. In
Section 3 we note that strong HKT manifolds which are not hyperk\"ahler do not admit a balanced hyperhermitian metric and consider the known constructions of strong HKT structures. In Section 4 we present the construction of  strong HKT and 
generalized hyperk\"ahler compact manifolds  via mapping tori.
\\

\textbf{Acknowledgements:} The authors would like to thank the anonymous referee for their suggestions and valuable comments which helped to improve this paper.

\section{HKT structures}

We start recalling the following

\begin{definition}
A {\em hypercomplex structure} on a $4n$-dimensional manifold $M$ is a triple $(I,J,K)$ of complex structures on $M$ satisfying the imaginary quaternionic relations.\\
Furthermore, if there exists a Riemannian metric $g$ with respect to which $I,J,K$ are skew, the structure $(I,J,K,g)$ is said hyperhermitian. 
\end{definition}
\begin{remark}
For quaternionic dimension $n=1$, compact manifolds admitting a hyperhermitian structure have been classified in \cite{Bo}. Any compact $4$-dimensional hyperhermitian manifold is either a torus with its flat metric, a K3 surface with its hyperk\"ahler metric, or the Hopf Surface with its standard locally conformally flat metric.
\end{remark}
Let us consider an Hermitian manifold $(M,I,g)$. The Levi-Civita connection is Hermitian, i.e., $\nabla^{LC} g=0, \ \nabla^{LC} J=0$, if and only if $d\omega=0$, where $\omega=gI$ is the fundamental form of $(I,g)$. Therefore, when $d\omega \neq 0$, one needs to modify the Levi-Civita connection to obtain a connection which is compatible with both $g$ and $I$. \\
In \cite{PG}, the affine line of Hermitian connections on the tangent bundle  $\operatorname{TM}$, known as Gauduchon or canonical connections, with expression
\begin{equation} \label{eqn:canonical}
g(\nabla^t_XY,Z)=g(\nabla^{LC}_XY,Z)+\frac{t-1}{4}(d^c\omega)(X,Y,Z)+\frac{t+1}{4}(d^c\omega)(X,IY,IZ),
\end{equation}
is introduced, where $d^c \omega=-Id\omega$. From now on, we will adopt the convention $Id\omega(X,Y,Z):=d\omega(IX,IY,IZ)$.
When $(M,I,g)$ is K\"ahler, $d^c\omega$ vanishes, so that the line collapses to a single point, namely, the Levi-Civita connection. Nonetheless, when  $(M,I,g)$ is not K\"ahler, the line is not trivial and the connections $\nabla^t$ have non vanishing torsion. 
For particular values of $t \in \R$, the Chern and the Bismut connections are recovered, i.e., $\nabla^{1}=\nabla_I^{Ch}$ and $\nabla^{-1}=\nabla_I^{B}$ (\cite{Bi, CH}). Despite $\nabla^{LC}, \nabla_I^{B}$, and $\nabla_I^{Ch}$ being mutually different connections, any one of them completely determines the other two. \\
In this work, we will focus mainly on the Bismut connection $\nabla^B$, also known as the Strominger connection \cite{Bi, Strominger}.  The Bismut connection can be characterized as the only Hermitian connection with totally skew-symmetric torsion and it follows from \eqref{eqn:canonical} that its expression is given by
\begin{equation*} 
g(\nabla^B_XY,Z)=g(\nabla^{LC}_XY,Z)-\frac{1}{2}d^c\omega(X,Y,Z),
\end{equation*}
with its torsion $3$-form $H$ given by
$$
H(X, Y,  Z) = g(T^B (X, Y), Z) = d \omega (IX, IY, IZ) = - d^c \omega (X, Y, Z), \quad  X, Y, Z \in \Gamma(TM).
$$
\begin{definition}
A $4n$-dimensional hyperhermitian manifold $(M,I,J,K,g)$ is said {\em hypek\"ahler with Torsion} (HKT, in short) if $\nabla^B_I=\nabla^B_J=\nabla^B_K=:\nabla^B$.
\end{definition}
The study of HKT structures is motivated by the fact that these structures pop up in string theory, in the context of certain supersymmetric sigma models. In fact, HKT manifolds first appeared in \cite{HP} as target spaces of $(4,0)$-SUSY sigma model with Wess-Zumino term. HKT structures also appear in supergravity. For instance, it has been shown that the geometry of the moduli space of a class of black holes in five dimensions is HKT  (\cite{pap}).
\begin{remark}
Recall that an Hermitian manifold $(M,J,g)$ is called Calabi-Yau with torsion (CYT in short) if the restricted holonomy of the Bismut connection is contained in $SU(n)$. This is equivalent to ask that the (first) Bismut Ricci curvature vanishes. \\
For an HKT manifold, $\operatorname{Hol}(\nabla^B) \subseteq Sp(n)$, since $\nabla^B g=0$ and $\nabla^B_I I=\nabla^B_J J=\nabla^B_K K=0$ (\cite{HP}). In particular, any of the three Hermitian structures is CYT.
\end{remark}
\begin{proposition} [\cite{GP}]\label{prop:HKT}
Given a hyperhermitian manifold $(M,I,J,K,g)$ the following  conditions are equivalent
\begin{enumerate}
\item $\nabla^B_I=\nabla^B_J=\nabla^B_K$,
\item $Id\omega_I=Jd\omega_J=Kd\omega_K$ \label{item:torsion},
\item $\partial_I(\omega_J+i\omega_K)=0$ \label{item:del},
\item $\overline{\partial}_I(\omega_J-i\omega_K)=0$.
\end{enumerate}
where ${\partial}_I$ and $\overline{{\partial}}_I$ denote the $\partial$ and $\overline{\partial}$ operators induced by the complex structure $I$.
\end{proposition} 
\begin{remark}
It has been shown in \cite{CS} that an almost hyperhermitian structure satisfying the condition (\ref{item:torsion}) in the previous proposition is, in fact, hyperhermitian.
\end{remark}
Given an hyperhermitian structure $(I,J,K,g)$ with fundamental forms $$\omega_I:=gI,  \quad \omega_J:=gJ, \quad  \omega_K:=gK, $$the form $\Omega:=\omega_J+i\omega_K$ is a $(2,0)$-form with respect to $I$ and one can easily observe that $J\Omega=\overline{\Omega}$ and that $\Omega(X,JX)>0$, for every non-zero vector field $X$. Conversely, any $(2,0)$-form with respect to $I$ satisfying the above conditions induces an hyperhermitian metric defined by $g(X,Y):=\Omega(X,JY)$ (with a slight abuse of notation, we will sometimes refer to $\Omega$ as the hyperhermitian metric). \\
By Proposition \ref{prop:HKT}, it is immediate to observe that in quaternionic dimension $1$, any hyperhermitian structure is HKT, since $\partial_I\Omega=0$. In higher dimension, given a hypercomplex structure, in general it is not true that there exists a compatible HKT metric. Some counterexamples arise on nilmanifolds of quaternionic dimension $n \ge 2$ (\cite{FG}). Here we briefly recall their construction. \\
\begin{definition} Let $\g$ be a Lie algebra endowed with a hypercomplex structure $(I,J,K)$. The hypercomplex structure is said to be {\em abelian}  if 
\[
[IX,IY]=[X,Y], \ [JX,JY]=[X,Y], \ [KX,KY]=[X,Y],
\]
for every $ X,Y \in \frak g.$
\end{definition}
It turns out that  any inner product compatible with an abelian hypercomplex structure is HKT. It has been shown in \cite{DF} that if $G$ is a $2$-step nilpotent Lie group, the hypercomplex structure arising from an invariant HKT one is abelian. As a consequence of the symmetrization process, one has the following.
\begin{theorem}[\cite{FG}]
Any compact quotient $M=\Gamma \backslash G$ of a $2$-step nilpotent Lie group with a non-abelian left invariant hyper-complex structure $(I,J,K)$ admits no HKT metric compatible with such a hypercomplex structure. 
\end{theorem}
\begin{example}
We exhibit an explicit counterexample. Let $G=\R \times H_7$ be the $2$-step nilpotent Lie group whose Lie algebra is defined by the structure equations
\[
[e_1,e_2]=[e_3,e_4]=-e_6, \ [e_1,e_3]=-[e_2,e_4]=-e_7, \ [e_1,e_4]=[e_2,e_3]=-e_8.
\]
It is straightforward to note that $G$ admits compact quotients by \cite{Ma}, since the structure equations of its Lie algebra are rational. The non-abelian hypercomplex structure is given by 
\[
\begin{split}
&I(e_1)=e_2, \ I(e_3)=e_4, \ I(e_5)=e_6, \ I(e_7)=e_8, \\
&J(e_1)=e_3, \ I(e_2)=-e_4, \ J(e_5)=e_7, \ J(e_6)=-e_8.\\
\end{split}
\]
Furthermore, it was also shown in \cite{FG} that $G$ admits a one parameter family of non-abelian invariant hypercomplex structures except for a unique value of the parameter, to which corresponds an abelian one.  Hence, the next Theorem is proved.
\end{example}
\begin{theorem} [\cite{FG}]
The HKT structure is not stable under deformations, i.e., there exists a small hypercomplex deformation of the HKT structure which is not HKT.
\end{theorem}
In a suitable sense, HKT metrics play in hypercomplex geometry the same role as K\"ahler metrics play in complex geometry. In fact, in analogy to the K\"ahler case, one has the existence of a local potential.
\begin{theorem}[\cite{BS}]  
Let  $(M,I,J,K)$ be a hypercomplex manifold. A form $\Omega \in \Omega^{2,0}_I(M)$ satisfying $J\Omega=\overline{\Omega}$ can be written locally as $\Omega=\partial_I \partial_J u$ for some smooth real-valued local function $u$, if and only if it is HKT, i.e., $\partial_I \Omega=0$. 
\end{theorem} 
Twistor theory for hypercomplex manifolds have been investigated in literature (\cite{sal, PP}). It is not surprising that when a hypercomplex manifold has a HKT-structure, the geometry of the twistor space is much more rich. In order to enunciate the main result, we need to fix some preliminary notations. \\
Let $(M, I,J,K)$ be a $4n$-dimensional hypercomplex manifold. The smooth manifold $Z = M  \times S^2$ admits an integrable complex structure defined as follows. For a unit vector $v = (v_1, v_2, v_3) \in \R^3$ let $I_v$ be the complex structure $v_1I + v_2 J + v_3K$ and let $J_v$ be the complex structure on $S^2$ defined by the cross product in $\R^3: J_v w = v \cross w$. \\The complex structure on $Z = M \times S^2$ at the point $(p,v)$ is $\cal{J}_{(p,v)} = I_v \oplus J_v$. We shall have to consider another almost complex structure, $\cal{J}_{2_{(p,v)}} = I_v \oplus -J_v$, which always fails to be integrable.  In the next theorem, we assume $Z$ endowed with the integrable complex structure $\cal{J}$.
\begin{theorem} [\cite{HP,GP}]
Let $(M,I,J,K)$ be a $4n$-dimensional HKT manifold. Then the twistor space $Z$ is a complex manifold such that
\begin{enumerate}
\item  the fibers of the projection $\pi : Z \to M$ are rational curves with holomorphic normal bundle $\oplus^{2n}\cal{O}(1)$,
\item there is a holomorphic projection $p : Z \to \CP^1$ such that the fibers are the manifold $M$ equipped with complex structures of the hypercomplex structure generated by $\{I,J,K\}$,
\item denoted by $\cal{D}$  the sheaf of kernel of the differential $dp$,  we have the exact sequence \[
0 \to  \cal{D} \to  \Theta_Z \to  p^*\overline{\Theta}_{\CP^1}  \xrightarrow{\text{dp}} 0,\] and there is a $\cal{J}_2$-holomorphic section of $\bigwedge^{(0,2)}\cal{D} \otimes p^*\overline{\Theta}_{\CP^1}$ defining a positive definite $(0,2)$-form on each fiber,
\item there is an anti-holomorphic map $\tau$ compatible with (1), (2) and (3) and inducing the antipodal map on $\CP^1$.
\end{enumerate}
Conversely, if $Z$ is a complex manifold with a non-integrable almost complex structure $\cal{J}_2$ with the above four properties, then the parameter space of real sections of the projection $p$ is a $4n$-dimensional manifold $M$ with a natural HKT-structure for which $ Z$ is the twistor space.
\end{theorem}
Another construction of HKT manifold is given by HKT reductions. This construction, originally developed by Joyce for hypercomplex manifolds (\cite{DJ}), has been improved for HKT manifolds in \cite{GPP}. We recall it here.\\
 Let $(M,I,J,K,g)$ be an HKT manifold and let $G$ be a compact group of hypercomplex isometries of $M$. Denote the algebra of hyper-holomorphic vector fields by $\g$. Suppose that $\mu = (\mu_1, \mu_2, \mu_3) : M \to \R^3 \otimes \g$ is a $G$-equivariant map satisfying: the Cauchy-Riemann condition, i.e.,  $Id\mu_1 = Jd\mu_2 = Kd\mu_3$, and the transversality condition, i.e., $Id\mu_1 \neq 0, \ Jd\mu_2 \neq 0, Kd\mu_3 \neq 0 $ for any $X \in \g$. Any map satisfying these conditions is called a $G$-moment map. 
\begin{theorem} [\cite{DJ,GPP}]
Let $(M,I,J,K,g)$  be a HKT manifold. Suppose that $G$ is a compact group of hypercomplex isometries admitting a $G$-moment map $\mu$. Then hypercomplex reduced space $N = M//G$  inherits a HKT structure.
\end{theorem}
\section{strong HKT structures}
Let  $(I, J, K, g)$ be a HKT structure on a 4n-dimensional manifold $M$. Setting $$H:= Id\omega_I=Jd\omega_J=Kd\omega_K,$$
we have that the HKT structure is said strong if $dH=0$ and weak otherwise. This is equivalent to require that any of the three Hermitian structures $(g,I), \ (g,J), \ (g,K)$ is SKT (or pluriclosed). Observe that when $H$ is identically zero, then $\nabla^B=\nabla^{LC}$ and the metric is hyperk\"ahler. \\
\begin{definition}
Let $(M,I,J,K,g)$ be a hyperhermitian manifold. $(I,J,K,g)$ is said {\em balanced hyperhermitian}  if any of the three Hermitian structures is balanced.
\end{definition} 
It turns out that the balanced hyperhermitian condition is orthogonal (in a suitable sense) to the strong HKT one. In fact, a natural generalization of the Fino-Vezzoni conjecture holds true in the hypercomplex setting.
\begin{theorem} \cite{FusiGentili}
Let $(M,I,J,K,g)$ be a compact non-hyperK\"ahler strong HKT manifold. Then there is no balanced hyperhermitian metric on $(M,I,J,K)$.
\end{theorem}
The original definition of HKT structures assumed the strong HKT condition; however in literature it was dropped because of the lack of examples. In fact, up to know, the only known examples of compact (non-hyperk\"ahler) strong HKT manifolds are due to Joyce (\cite{Joyce}) and Barberis-Fino (\cite{Bafi}). Moreover, all of them are homogeneous. \\
In the non-compact case, a non-homogeneous example has been constructed in \cite{MV} in the following way
\begin{theorem}
Let $(M,I,J,K,g)$ be a compact strong HKT-manifold of real dimension $4$, and $E$ be a smooth complex vector bundle on $M$. De-note by $\cal{M}$ the moduli space of gauge-equivalence classes of anti-self-dual connections (instantons) on $E$. Then $\cal{M}$ is equipped with a natural strong HKT-structure.
\end{theorem}
Except for the $4$-dimensional case, where each HKT manifold is strong HKT, very few is known about the geometry of strong HKT manifolds. \\
 In the locally homogeneous setting, HKT structures on manifolds have been studied in \cite{BDV,DF2,DF3,AB}. In particular, it is known that a nilmanifold and a almost abelian solvmanifold admit strong HKT structures if and only if they are hyper-K\"ahler (\cite{BDV,DF,AB}), but a general result on solvmanifolds is still unknown. \\
In the following sub-sections we describe the constructions of Joyce (\cite{Joyce}) and Barberis-Fino (\cite{Bafi}) in details. We recall that for both the examples the Bismut connection is flat. In the last section of this work, we will construct an example of a (non-homogeneous) compact manifold admitting a non Bismut flat strong HKT structure.
\subsection{Joyce's construction}
Let $G$ be a compact Lie group, which after a covering argument, we may assume to be semisimple. \\
Let then $G$ be a compact semisimple Lie group and let $H \subseteq G$ be a maximal torus in it. 
\begin{theorem} \cite{Joyce}
The Lie algebra $\g$ of $G$ decomposes as
\[
\g= \mathfrak{b} \oplus \bigoplus_{j=1}^m \mathfrak{d}_j  \oplus \bigoplus_{j=1}^m \mathfrak{f}_j 
\]
where $\mathfrak{b}$ is abelian and of dimension $\operatorname{rank} (G)-m$, $\mathfrak{d}_j \cong \mathfrak{su}(2)$ and $\mathfrak{f}_j$ satisfy the following conditions
\begin{enumerate}
\item $[\mathfrak{d}_j,\mathfrak{b}]=0$ and $\mathfrak{b} \oplus \bigoplus_{j=1}^m \mathfrak{d}_j $ contains the Lie algebra of $H$ 
\item $[\mathfrak{d}_j,\mathfrak{d}_i]=0$ for $i \neq j$
\item $[\mathfrak{d}_j,\mathfrak{f}_i]=0$ for $j < i$
\item $[\mathfrak{d}_j,\mathfrak{f}_j] \subseteq \mathfrak{f}_j$ where the Lie bracket action of $\mathfrak{d}_j$ on $\mathfrak{f}_i$ is isomorphic to the sum of a finite amount of copies of the action of $\mathfrak{su}(2)$ on $\C^2$ by left multiplication.
\end{enumerate}
\end{theorem}
Let $\torus^{2m-r}$, with $r=\operatorname{rank} (G)$, be the flat torus, so that the Lie algebra of $\torus^{2m-r} \times G$ decomposes as
\[
\R^{2m-r} \oplus \g= \R^m \oplus \bigoplus_{j=1}^m \mathfrak{d}_j  \oplus \bigoplus_{j=1}^m \mathfrak{f}_j.
\]
Let $(e_1,\dots,e_m)$ be a basis of $\R^m$ and ket $\varphi_j$ be the Lie algebra isomoprhism $\varphi: \mathfrak{su}(2) \to \mathfrak{d}_j$, where we assume that $\mathfrak{su}(2)$ is endowed with the basis $(i_1,i_2,i_3)$ satisfying \[[i_1,i_2]=2i_3, \ [i_2,i_3]=2i_4, \ [i_3,i_4]=2i_1.\] 
The definition of the hypercomplex structure is the following \\
\begin{enumerate}
\item $I_1,I_2,I_3$ act on $\R^m \oplus \bigoplus_{j=1}^m \mathfrak{d}_j $ as
\[ I_a(e_j)=\varphi_j(i_a), \ \ I_a(\varphi_j(i_a)=-e_j, \ \ I_a(\varphi_j(i_b)=\varphi_j(i_c), \ \ I_a(\varphi_j(i_c)=-\varphi_j(i_b),\] 
where $(a,b,c)$ is an even permutation of $(1,2,3)$. \\
\item $I_1,I_2,I_3$ acts on $\mathfrak{f}_j$ as $I_a(v)=[\varphi_j (i_a),v]$, for any $v \in \mathfrak{f}_j$.
\end{enumerate}
By definition is clear that $(I_1,I_2,I_3)$ is an hypercomplex structure on $\torus^{2m-r} \times G$, which by \cite{sam} is integrable. Actually, all the invariant hypercomplex structures on compact Lie groups are obtained by this construction.\\
Furthermore, the (opposite) of the Killing Cartan form of $\g$ extends to a hyperhermitian bi-invariant metric $b$ on $\torus^{2m-r} \times G$ (\cite{GP}), which is HKT. In fact, using that $I_a$ are integrable complex structures and $b$ is bi-invariant, one gets that
\[
d^c_a\ \omega_a (X,Y,Z)=-b([X,Y],Z),
\]
for any $X,Y,Z$ left invariant vector fields. In addition, since $d^c_a\omega_a$ are bi-invariant $3$-forms on $\mathbb{T}^{2m-r}\times G$, they must be closed, implying that the HKT structure $(I_a,b)$ is strong HKT.  The remark that the metric is HKT is originally due to Opfermann and Papadopoulos \cite{OP}, who also generalized the construction to certain homogeneous spaces, showing that they carry HKT structures.
\subsection{Barberis-Fino construction}
Let $\g$ be a $4n$-dimensional Lie algebra endowed with an hypercomplex structure $(I,J,K)$ and let $\rho:\g \to \mathfrak{gl}(k,\quat)$ be a Lie algebra homomorphism. $T_\rho \g:=\g \ltimes_\rho \quat$ has the structure of Lie algebra with Lie bracket
\[
[(X,U),(Y,V)]=[([X,Y]),	\rho_X(V)-\rho_Y(U)],
\]
where $X,Y \in \g$ and $U,V \in \quat^k$.\\
The hypercomplex structure on $\g$ induces an hypercomplex structure on $T_\rho \g$ 
\[
\tilde I (X,U):=(IX, iU), \  \ \tilde J (X,U):=(JX, jU), \ \  \tilde K (X,U):=(KX, kU),
\]
where $i,j,k$ are the quaternionic units of $\quat$. Furthermore, given an hyperhermitian metric $g$ on $\mathfrak{g}$, $g$ and the standard metric on $\quat^k$ induce a hyperhermitian metric $\tilde g$ on $T_\rho \g$ in such a way $\g$ and $\quat^k$ are orthogonal.
\begin{theorem}[\cite{Bafi}]
When $\rho:\g \to \mathfrak{sp}(k)$, the Lie algebra $(T_\rho \g, \tilde I, \tilde J, \tilde K, \tilde g)$ is strong HKT if and only if $(\g,I,J,K,g)$ is strong HKT.
\end{theorem}
We exhibit an explicit example. By Joyce's construction the Lie algebra $\mathfrak{sp}(1) \oplus \R$ admits a strong HKT structure. Let $\rho:\R \oplus\mathfrak{sp}(1) \to \mathfrak{sp}(1)$ be a Lie algebra homomorphism defined as follows: fixed the standard basis $\{e_1,e_2,e_3\}$ basis of $\mathfrak{sp}(1)$ and a generator $e_4$ of $\R$,
\[
\rho(e_1):=\frac{1}{2}\begin{pmatrix}
0 & 0 & 0 & -1 \\
0 & 0 & -1 & 0 \\
0 &1 & 0 & 0 \\
1 & 0 & 0 & 0 \\
\end{pmatrix}, \ \rho(e_2):=\frac{1}{2}\begin{pmatrix}
0 & 0 & 1 & 0 \\
0 & 0 & 0 & -1 \\
-1 &0 & 0 & 0 \\
0 & 1 & 0 & 0 \\
\end{pmatrix}, \ 
 \rho(e_3):=\frac{1}{2}\begin{pmatrix}
0 & -1 & 0 & 0 \\
1 & 0 & 0 & 0 \\
0 &0 & 0 & -1 \\
0 & 0 & 1 & 0 \\
\end{pmatrix}, 
\]
and $\rho(e_4):=0$.\\
As a consequence of the Theorem, the $8$-dimensional Lie algebra $T_\rho \g$ with structure equations 
\[
\begin{split}
&[e_1,e_2]=e_3, \ [e_2,e_3]=e_1, \ [e_3,e_1]=e_2, \\
&[e_1,e_8]=-[e_2,e_7]=[e_3,e_6]=\frac{1}{2} e_5, \\
&[e_1,e_7]=[e_2,e_8]=-[e_3,e_5]=-\frac{1}{2} e_6,\\
&[e_1,e_6]=-[e_2,e_5]=-[e_3,e_8]=\frac{1}{2} e_7,\\
&[e_1,e_5]=[e_2,e_6]=-[e_3,e_7]=\frac{1}{2} e_8,
\end{split}
\]
admits a strong HKT structure. The corresponding connected and simply connected Lie group is $ SU(2) \ltimes \R^4 \times \R$. It has been shown in \cite{FT} that the $7$-dimensional Lie group $SU(2) \ltimes \R^4$ admits a compact quotient $M^7$, giving an  example of  strong HKT  $8$-dimensional manifold $M^7 \times S^1$. We recall here the construction of the lattice. \\
Let $\beta$ be a $p$-th rooth of the unity with $p$ prime and let us consider the discrete subgroup of $SU(2)$ generated by the matrix
\[
A_\beta:=\begin{pmatrix}
\exp(i\beta) & 0 \\
0 & \exp(-i\beta)
\end{pmatrix}.
\] 
The set $\Gamma_\beta:=\langle A_\beta \rangle \ltimes \Z^4$ is then a closed and discrete subgroup of $SU(2) \ltimes \R^4$. In particular, fixed any point $(A',q')$ of $SU(2) \ltimes \R^4$, then $[(A',q')]=[(A',q'+r)]$ where $r \in \Z^4$, and so the restriction of $\pi: SU(2) \ltimes \R^4 \to (SU(2) \ltimes \R^4)/ \Gamma_\beta$ to  $SU(2) \times [0,1]^4$ is surjective. This proves that $(SU(2) \ltimes \R^4)/ \Gamma_\beta$ is compact. \\
 We point out that this example, and each example constructed in this way, is Bismut flat. Indeed, the strong HKT metric on the universal cover is given by $\sum_{i=1}^8 (e^{i})^{2}$, which is not bi-invariant on  $SU(2) \ltimes \R^4 \times \R$, but it is bi-invariant on $SU(2) \times \R \times \R^4 $ with the standard Lie group structure.

\section{Construction of strong HKT manifolds via mapping tori}
We start by recall the following
\begin{definition}
Let $M$ be a smooth manifold and let $f \in \operatorname{Diff}(M)$. The  {\em mapping torus}  $M_f$ is the smooth manifold 
\[
\bigslant{M\times \R} { \Z},
\]
where the action of $\Z$ on $M\times \R$ is given by $n \cdot (p,x)=(f^n (p),x+n)$, for any $n \in \Z, \ p \in M$ and $ x \in \R$. 
\end{definition}
By \cite{TS}, any compact manifold endowed with a non vanishing closed $1$-form has the structure of a mapping torus. In particular, any solvmanifold is a mapping torus. 
\begin{definition} [\cite{Li}]
A {\em co-K\"ahler} manifold is a mapping torus $M_\psi$, where $(M,J,g)$ be a K\"ahler manifold, and $\psi$ is a K\"ahler isometry of $(J,g)$.
\end{definition}
Co-K\"ahler manifolds are the odd counterpart of K\"ahler manifolds, as they share some topological properties, e.g. for instance the formality in the sense of Sullivan \cite{CLM}. \\
When a co-K\"ahler manifold is obtained by a mapping torus of an hyperk\"ahler manifold $(M,I,J,K,g)$ via an hyperk\"ahler isometry, then it is said a co-hyperk\"ahler manifold. \\
It is well known that when the hyperk\"ahler manifold is a K3 surface, the hyperk\"ahler metric is in general not explicit. Henceforth, determining an explicit non-trivial hyperk\"ahler isometry is challenging. Here we exhibit an explicit example.\\
\begin{example} \label{example:hkisometry}
The Fermat quartic is the complex surface in $\CP^3$ defined by the equation 
\[
\cal{F}=\{z_0^4+z_1^4+z_2^4+z_3^4=0\},
\]
where $z_0,z_1,z_2,z_3$ are the standard homogeneous coordinates on $\CP^3$. It is well known, that it is a K3 surface. We will always assume that  $\cal{F}$ is endowed with the standard complex structure $I$, i.e., $I$ is such that $\iota: \cal{F} \to \CP^3$ is holomorphic, where $\iota$ is the natural embedding. Let $\iota^*\omega_{FS}$ be the K\"ahler metric on $\cal{F}$ induced by the Fubini-Study metric on $\CP^3$ and define $g$ to be the unique K\"ahler Ricci flat metric on $\cal{F}$ whose K\"ahler form in cohomologous to $\iota^*\omega_{FS}$, which is hyperk\"ahler. We will denote by $(I,J,K,g)$ the hyperk\"ahler structure. \\
The explicit form of $g$ is not known, nevertheless, in \cite{AG} the group of (not necessarily holomorphic) isometries of $g$ is characterized. In particular, $\operatorname{Isom}(g)$ is identified with the group of all holomorphic and antiholomoprhic isometries of $\CP^3$ which preserve $\cal{F}$. \\
Let $\sigma:\CP^3 \to \CP^3$ be defined by $[z_0,z_1,z_2,z_3] \mapsto [-z_0,-z_1,z_2,z_3]$. By the identification above $\psi:=\sigma_ {|\cal{F}}$ is clearly an isometry of $(\cal{F},g)$, and it is straightforward to prove that $\psi$ is homolorphic with respect to $I$.  Hence, $\psi \in \operatorname{Isom}_{hol}(g)$. It remains to prove that $\psi$ is an hyperk\"ahler isometry, i.e., that $\psi$ preserves the fundamental forms $\omega_J=gJ$ and $\omega_K=gK$. \\
Since $\cal{F}$ is a $K3$ surface, up to complex scalar multiples, there exists a unique $(2,0)$ holomorphic form $\Omega$ on $\cal{F}$. The fixed points of $\sigma$ in $\CP^3$ are the lines $z_0=z_1=0$ and $z_2=z_3=0$, and any of them intersects the Fermat quartic in four isolated fixed points. Since the fixed loci of $\psi$ is given by eight isolated and distinct points, $\psi$ is symplectic by \cite{Nik}. In fact, since $\psi$ is a holomorphic involution of $\cal{F}$ then either it preserves $\Omega$ or takes $\Omega$ in $-\Omega$. So, it suffices to observe that the second case leads to a contradiction. \\
Let $p$ be any of the eight fixed points of $\psi$. Then the eigenvalues of $d\psi$ at $T_p \cal{F}$ are $\pm 1$, since it is a involution. However, if there is an eigenspace with eigenvalue $1$, then the geodesics tangent to this eigenspace would be also fixed by $\psi$, since it is an isometry. But this contradicts the fact that the fixed points are isolated. Therefore, $d\psi=-Id$ on $T_p \cal{F}$ and $\psi^* \Omega=\Omega$. \\
Let $k \in \C^*$ be such that $\abs{k}^2 \Omega \wedge \overline{\Omega}=2 \omega_I^2$. Then, $\omega_J=\operatorname{Re}(k\Omega)$ and  $\omega_K=\operatorname{Im}(k\Omega)$. Since $\psi^*\Omega=\Omega$, $\psi^* (k\Omega)=k\Omega$.\\
To conclude the proof, we observe that since $\psi$ is holomorphic
\[
\operatorname{Re}(k\Omega)+i \operatorname{Im}(k\Omega)=k\Omega=\psi^*(k\Omega)=\psi^*(\operatorname{Re}(k\Omega)+i \operatorname{Im}(k\Omega))=\psi^*\operatorname{Re}(k\Omega)+i \psi^*\operatorname{Im}(k\Omega), \\
\]
from which follows that $\psi^*\omega_J=\omega_J$ and $\psi^*\omega_K=\omega_K$.
\end{example}
The mapping torus construction turns out to be particularly adaptable to provide examples of complex manifolds endowed with special Hermitian metrics (\cite{BF, BFG,FGV}).  In the next Theorem we show that the product of a co-hyperk\"ahler manifold with the $3$-sphere $S^3$ admits a strong HKT metric. 
\begin{theorem} \label{theorem:shkt}
Let $(N,I,J,K,g)$ be a compact hyperk\"ahler manifold let $\psi$ be a hyperk\"ahler isometry. Then, the product $ N_\psi \times S^3$ admits a strong HKT structure.
\end{theorem}
\begin{proof}
We first observe that $M_\psi \times S^3$ is diffeomorphic to the mapping torus $M_f=(N \times S^3)_{(\psi,Id)}$. On the other hand, $M_f \cong \bigslant{N\times \C^2\setminus\{(0,0)\} }{ \Z} $, where  the action of $\Z$ on $N \times \C^2 \setminus \{(0,0)\}$ is given by
\begin{equation} \label{eqn:Zaction}
 n \cdot (p, \underline{x}) \mapsto (\psi^n (p), 2^n \underline{x}).
\end{equation}
The identification above is provided by the map 
$$
\alpha: N \times \Sph^3 \times \R \to N \times \C^2 \setminus \{(0,0)\}, \quad 
           (p, q, t)  \mapsto (p, 2^t \cdot q)
$$
with inverse
$$
\alpha^{-1}: N \times \C^2 \setminus \{(0,0)\} \to N \times \Sph^3 \times \R, \quad 
           (p, \underline{x}) \mapsto (p, \frac{\underline{x}}{\norm{\underline{x}}}, \log_2 \norm{\underline{x}})
$$
induced on the quotients. \\
Hence, it suffices to prove that $\bigslant{N \times \C^2\setminus\{(0,0)\} }{ \Z} $ admits a strong HKT structure. 
Let us fix on $N \times \C^2 \setminus \{(0,0)\}$ the product hypercomplex structure $(I \times I_-, J \times J_-, K \times K_-)$, where $I_-$ is the standard complex structure on $\C^2 \setminus \{(0,0)\}$ with associated holomorphic coordinates $(z^1,z^2)$, $J_-$ is defined by $J_- (dz^1)=-d\overline{z}^2$ and $K_-=I_-J_-$. \\
We consider the product metric $g+b$, where $b:=\frac{1}{\abs{z}} (dz_1d \overline{z}_1 +  dz_2  d \overline{z}_2)$. It is well known that $(I_-,J_-,K_-, b)$ is a strong HKT structure on  $ \C^2\setminus\{(0,0)\} $, and therefore, $(I \times I_-,J \times J_-,K \times K_-, g+b)$ is a strong HKT structure on $N \times  \C^2\setminus\{(0,0)\} $. \\
Finally, since the automorphisms $\phi_n$ associated to the action \eqref{eqn:Zaction} are manifestly holomorphic with respect to each complex structure and they preserve the product metric $g+g_-$, the strong HKT structure descend to the quotient, concluding the proof of the Theorem.
\end{proof}
We recall the following definition
\begin{definition}
Let $M$ be a $4n$-dimensional manifold endowed with a pair of HKT structures $(I,J,K,g)$ and $(I',J',K',g)$. The structure $(I,J,K,I',J',K',g)$ is said to be \emph{generalized hyperk\"ahler} if the Bismut torsions $H$ and $H'$ of $(I,J,K,g)$ and $(I',J',K',g)$ satisfy $H=-H'$ and $dH=dH'=0$. 
\end{definition}
\begin{proposition}
The manifold $N_\psi \times S^3$ admits a generalized Hyperk\"ahler structure. 
\end{proposition}
\begin{proof}
The proof proceeds as the one of Theorem \ref{theorem:shkt}. In fact, we observe that the hypercomplex structure $(I_-,J_-,K_-)$ on $\C^2 \setminus\{(0,0)\}$ is precisely the left multiplication by $i, j$, and $k$ on $\mathbb{H} \setminus \{0\}$.  If we consider the other hypercomplex structure $(I_+,J_+,K_+)$ on  $\mathbb{H} \setminus \{0\}$  corresponding to the right multiplication by $i, j$, and $k$, then $(I_-,J_-,K_-,I_+,J_+,K_+,b)$ is a generalized hyperk\"ahler structure on $\C^2 \setminus\{(0,0)\}\cong \mathbb{H} \setminus \{0\}$ (see \cite{MV} for the details). 
Therefore the  structure $(I \times I_-,J \times J_-,K \times K_-,I \times I_+,J \times J_+,K \times K_+,g+b)$ is a generalized hyperk\"ahler on $K \times \C^2 \setminus\{(0,0)\}$, which descends to the quotient $N_\psi \times S^3$ by the definition of the action \eqref{eqn:Zaction}. In fact, on the hyperK\"ahler factor, $\Z$ acts by hyper-holomorphic isometries, and on $ \C^2 \setminus\{(0,0)\} \cong \mathbb{H} \setminus \{0\}$ it acts by multiplication by $2$, which clearly commute with the action of right multiplication by $i, j$, and $k$, concluding the proof.
\end{proof}
In the next Proposition we give some topological properties of  $N_\psi \times S^3$. 
\begin{proposition} \label{proposition:prop}
Let $ N_\psi \times S^3$ be a manifold constructed as in Theorem \ref{theorem:shkt}. Then
\begin{enumerate}
\item $N_\psi \times S^3$ is formal; \\
\item $N_\psi \times S^3$ is non-K\"ahler.
\end{enumerate}
\end{proposition}
\begin{proof} 
As $N_\psi$ is a co-K\"ahler manifold, it is formal. The results follows by the fact that the formality is preserved under the cartesian product. \\
The second statement is a consequence of the K\"unneth formula. Indeed, by the K\"unneth formula, $b_1(N_\psi \times S^3)=b_1(N_\psi)$. Since the first Betti number of a co-K\"ahler manifold is odd \cite{CLM}, the result easily follows. 
\end{proof}
\begin{example} \label{ex:1}
Let $N=\cal{F}$ be the Fermat quartic, and let $\psi$ be the hyperk\"ahler isometry constructed in the Example \ref{example:hkisometry}. By Theorem \ref{theorem:shkt}, the manifold $N_\psi \times S^3$ admits a strong HKT structure. We show that the Example is not trivial, i.e., it does not split as a product of $N \times S^3 \times S^1$. Since $\psi \neq Id$, the result follows by \cite{BFG}. We also observe that it is a non-trivial example of a non-K\"ahler Bismut Hermitian Einstein manifold. \end{example}

\begin{remark}
We note the difference of the geometry induced on a mapping tori in the case when $\psi$ is a holomorphic symplectic automorphism of infinite order. According to \cite{FGV}, the space $N_\psi\times S^1$ admits a balanced metric and no SKT metric. 
We point out that the example constructed in Example \ref{ex:1} admits a finite cover isomorphic to $\cal{F} \times S^3 \times S^1$, according to \cite{BPT}. In fact, since any hyperK\"ahler isometry is of finite order, there must exists a $k \in \Z$ such that $\psi^k =Id$, then $\tilde X= \cal{F} \times S^3 \times \R / k\Z$ where $k\Z \subset \Z$ acts by $kn \cdot (q,p,t)=(\psi^{kn} q, p, t+kn)$ is a finite cover of $N_\psi \times S^3$ and it is isomorphic to $\cal{F} \times S^3 \times S^1$.\\
 In particular, our example $N_\psi \times S^3$ is a local product of two strong HKT manifolds. Indeed,  it admits a  $\mathbb{Z}_2$ quotient which is a product of a Hopf surface and a K3 orbifold $N / \langle \psi \rangle$.
\end{remark}

\end{document}